\documentclass[11pt]{article}
\input epsf
\usepackage{amsfonts}
\usepackage{amssymb}
\usepackage{amsthm}
\newtheorem{theorem}{Theorem}[section]

\newtheorem{lemma}[theorem]{Lemma}
\newtheorem{proposition}[theorem]{Proposition}

\newtheorem{remark}{Remark}[section]

\newtheorem{definition}{Definition}[section]

\def\R{{\rm I}\!{\rm R}}

\def\L2{L^2(\Omega)}
\def\h01{H^1_0(\Omega)}
\def\h10{H^1_0(\Omega)}

\def\a{\alpha}

\def\p{{\bf p}}
\def\q{{\bf q}}
\def\n{{\mathbf n}}

\def\u{{\bf u}}
\def\w{{\bf w}}

\def\x{{\bf x}}
\def\y{{\bf y}}

\def\vv{{\bf v}}

\def\d{\mathrm{div\,}}

\def\RT{{\mathcal RT}}

\def\l{\ell}
\def\W1p{W^{1,p}}
\def\Lp{L^p}
\def\hT{\widehat T}
\def\tQ{\widetilde Q}
\def\btQ{\widetilde{\bf Q}}
\def\tu{\widetilde{\bf u}}
\def\tT{\widetilde T}

\begin{document}
%\date{}
\title{Error estimates for  Raviart-Thomas interpolation of any order
on anisotropic tetrahedra}

\footnotetext[1]{Departamento de Matem\'atica, Facultad de
Ciencias Exactas y Naturales, Universidad de Buenos Aires, 1428
Buenos Aires, Argentina.}

\footnotetext[2]{Institut f\"ur Mathematik und Bauinformatik,
Universit\"at der Bundeswehr M\"unchen, Neubiberg, Germany.}

\footnotetext[3]{Instituto de Ciencias, Universidad Nacional de
General Sarmiento, J.M. Gutierrez 1150, Los Polvorines, B1613GSX
Provincia de Buenos Aires, Argentina.}

\footnotetext[4]{CONICET, Argentina.}

\author{G. Acosta$^1$ \and  Th. Apel$^2$   \and R. G. Dur\'an$^{1,4}$
\and A. L. Lombardi$^{1,3,4}$}

\date{}
\maketitle

\noindent{\small{\bf Abstract.} We prove optimal order error estimates
for the Raviart-Thomas interpolation of arbitrary order under
the maximum angle condition for triangles and under two generalizations of
this condition, namely, the so-called
three dimensional maximum angle condition and the regular vertex property,
for tetrahedra.

Our techniques are different from those used in previous papers
on the subject and the results obtained are more general in several
aspects.
First, intermediate regularity is allowed, that is, for the Raviart-Thomas
interpolation of degree $k\ge 0$, we prove error estimates
of order $j+1$ when the vector field being approximated has components in
$W^{j+1,p}$, for triangles or tetrahedra,
where $0\le j \le k$ and $1\le p \le\infty$. These results are new
even in the two dimensional case. Indeed, the estimate was known only
in the case $j=k$. On the other hand, in the three dimensional case,
results under the maximum angle condition were known only for $k=0$.

\bigskip

\bigskip


\section{Introduction} The Raviart-Thomas finite element spaces  were
introduced in \cite{RT,T}, and extended to the three-dimensional case by
N\'ed\'elec \cite{N}, to approximate second order elliptic problems in mixed form.
After publication of that paper there has been an increasing
interest in the analysis of these spaces and on the approximation
properties of the associated Raviart-Thomas interpolation
operator. This interest has been motivated by the fact that, apart
from the original motivation, these spaces (or rotated versions of
them in the two dimensional case) arise in several interesting applications. For example in
mixed methods for plates (see \cite{BF,BFS,DL}) and
in the numerical approximation of fluid-structure interaction
problems  \cite{BDMRS}. Also, it is well known that mixed methods
are related to non-conforming methods \cite{AB,M}, therefore, the
Raviart-Thomas interpolation operator can be useful in some cases
to analyze this kind of methods (see for example \cite{AD1} where
a non-conforming method for the Stokes problem is analyzed).

The original error analysis developed in \cite{N,RT,T} is based on the
so-called regularity assumption on the elements and therefore, the constants
arising in the error estimates in those works depend on the ratio between
outer and inner diameter of the elements. In this way narrow or anisotropic elements,
which are very important in many applications, are excluded.

For the standard Lagrange interpolation it is known, since the pioneering works
\cite{BA,J,Synge} and many generalizations of them (see \cite{Ap} and its references),
that the regularity assumption can be relaxed to a \emph{maximum angle condition} in many cases.

Error estimates for the Raviart-Thomas interpolation under conditions weaker
than the regularity have been proved in several papers. In \cite{AD1} the lowest
order case $k=0$ was considered and optimal order error estimates were proved
under the maximum angle condition for triangles and a suitable generalization
of it for tetrahedra, called \emph{regular vertex property}. This result was extended in \cite{FNP} to prismatic elements and functions from weighted Sobolev spaces.

It is not straightforward to extend the arguments given in \cite{AD1}
to higher order Raviart-Thomas approximations. In \cite{D2} it was proved
that the maximum angle condition is also sufficient to obtain optimal
error estimates for the case $k=1$ and $n=2$ and in \cite{DL3} that result was
generalized to any $k\ge 0$. Also in \cite{DL3}, error estimates for any $k\ge 0$
and $n=3$ were proved assuming the regular vertex property.

The error estimates obtained in \cite{DL3}
require ``maximum regularity''. To be precise let $\Pi_k\u$ be
the Raviart-Thomas interpolation of degree $k$ of $\u$ on a triangle $T$ then, it was proved
in \cite{DL3} that
\begin{equation}
\label{DL3}
\|\u-\Pi_k\u\|_{L^2(T)}
\le\frac{C}{\sin\a}\,\,h_T^{k+1} |\u|_{H^{k+1}(T)}
\end{equation}
where we have used the standard notation for Sobolev seminorms,
$\a$ and $h_T$ are the maximum angle and the diameter of $T$ respectively,
and the constant $C$ is independent of $T$. However, an estimate like (\ref{DL3}) but
with $k$ replaced by $j<k$, only on the right hand side, cannot be proved by the arguments given in \cite{DL3}
and therefore a different approach is needed. Let us remark that this kind of estimates
is important in many situations. In particular, the lowest order estimate
$$
\|\u-\Pi_k\u\|_{L^2(T)}
\le\frac{C}{\sin\a}\,\,h_T |\u|_{H^1(T)}
$$
is fundamental in the error analysis for the scalar variable in mixed
approximations of second order elliptic problems. In particular
the inf-sup condition can be obtained from this estimate (see for example \cite{DR,D3}).

The maximum angle condition was originally introduced for triangles.
For the three dimensional case two different generalizations have been given.
One is the K\v r\'\i\v zek maximum angle condition
introduced in \cite{K}: the angles between faces and the angles in the faces are bounded
away from $\pi$.
Another possible extension is the regular vertex property
introduced in \cite{AD1}: a family of tetrahedral elements satisfies this
condition if for each element there is at least one vertex such that
the unit vectors in the direction of the edges sharing that vertex
are ``uniformly'' linearly independent,
in the sense that the volume determined by them is uniformly bounded
away from zero.

These two conditions are equivalent in two dimensions but not
in three. Indeed, the K\v r\'\i\v zek maximum angle condition
allows for more general elements. This can be seen in the following way:
consider the two families of elements
given in Figure \ref{prismas}, where $h_1$, $h_2$ and $h_3$ are arbitrary positive numbers.
Both families satisfy the K\v r\'\i\v zek condition
but the second family
does not satisfy the regular vertex property.

\begin{figure}[h]
\begin{center}
\label{prismas}
\begin{tabular}{cc}
\epsfxsize 5  cm \epsfbox{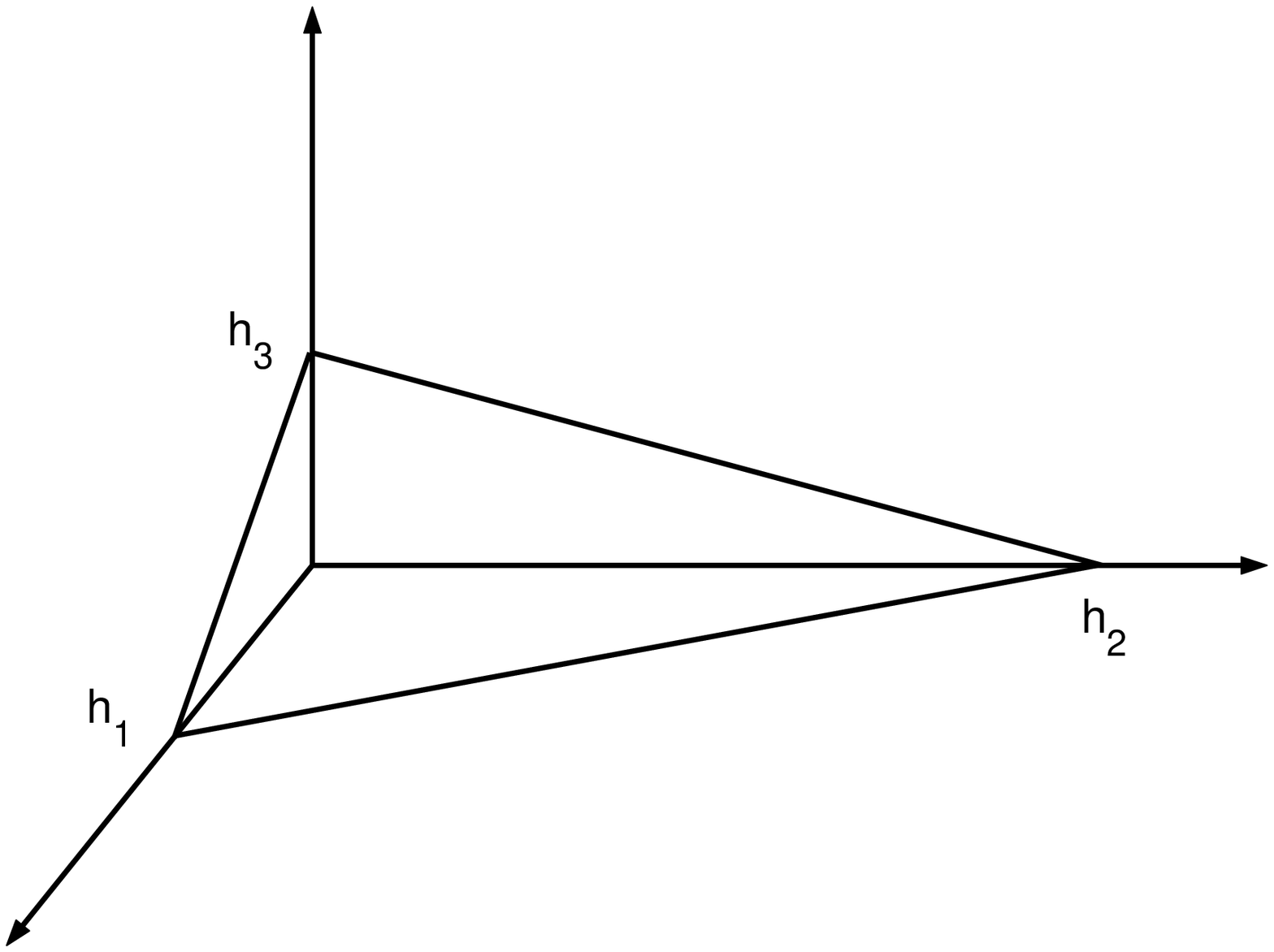} & \epsfxsize 5 cm
\epsfbox{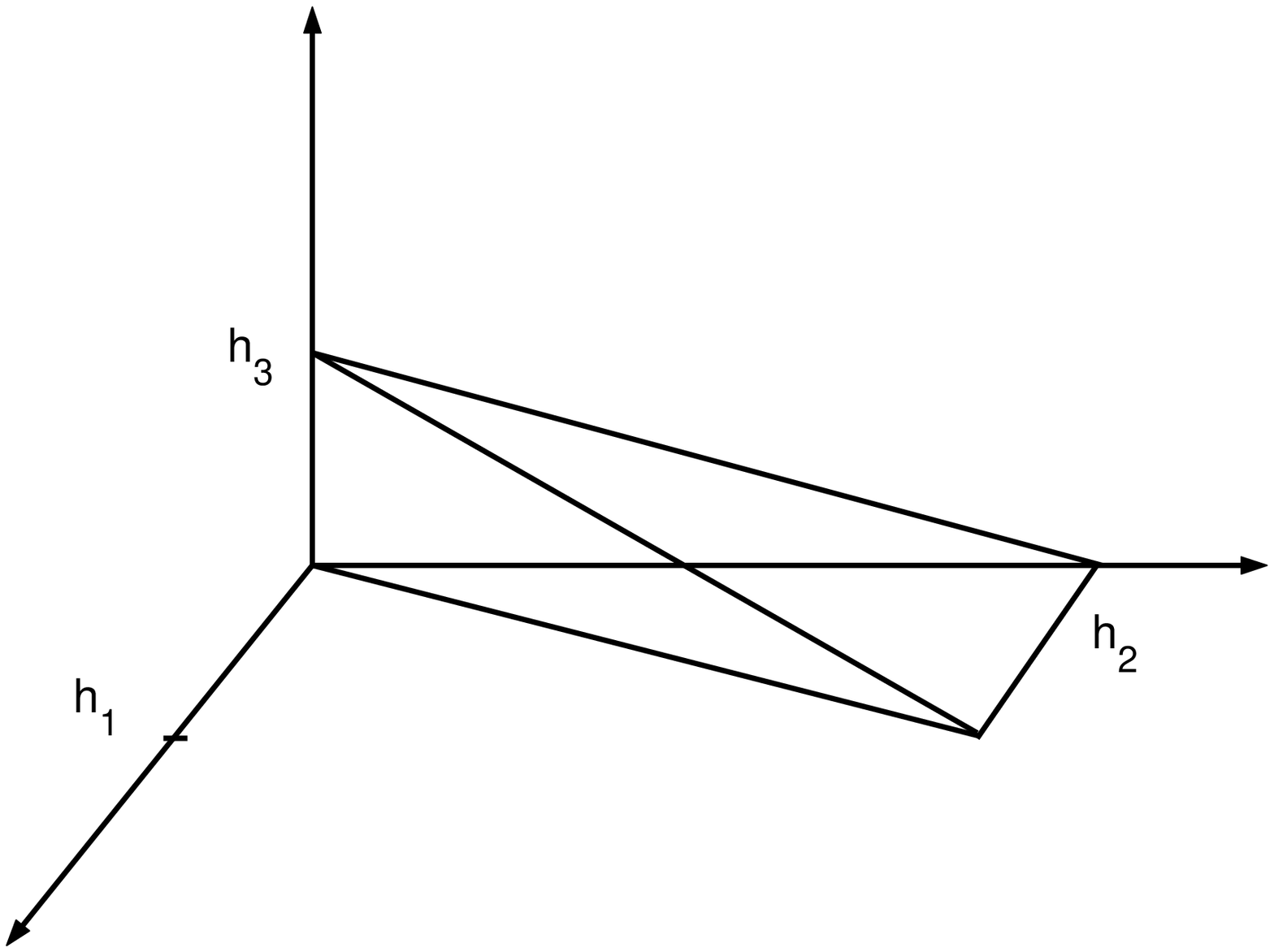}
 \\ (a) & (b)
\end{tabular}
\\ Figure 1
\end{center}
\end{figure}

Essentially these two families of elements give all possibilities.
Indeed, the family of all elements satisfying the K\v r\'\i\v zek
condition with a constant $\bar\psi<\pi$ (i.e., angles between
faces and angles in the faces less than or equal to $\bar\psi$)
can be obtained transforming both families in the figure by
``good'' affine transformations (see Theorem \ref{mac} for the
precise meaning of this). This result was obtained in \cite{AD1}
in the proof of Lemma 5.9. For the sake of clarity we will include
this result as a theorem. On the other hand, the family of all
elements satisfying the regular vertex property with a given
constant (see Section 3 for the formal definition of this) is
obtained by transforming in the same way only the first family in
the figure.
Therefore, to obtain general results under the K\v r\'\i\v zek maximum angle condition
({\sl resp. regular vertex property})
it is enough to prove error estimates for both families ({\sl resp. the first family})
in Figure \ref{prismas} with constants independent of the relations
between $h_1$, $h_2$ and $h_3$.

The error estimates in \cite{DL3} for the general $\RT_k$ were obtained
assuming the regular vertex property and the arguments given in
that paper cannot be extended to treat the more general case of
elements satisfying the K\v r\'\i\v zek condition. On the other hand,
as we have mentioned above, the arguments in \cite{DL3} can not
be applied to obtain error estimates for functions in $H^{j+1}(T)^n$
with $j<k$. For these reasons we need to introduce here a different approach.

In this paper we complete the error analysis for the Raviart-Thomas interpolation
of arbitrary order $k\ge 0$.
We develop the analysis in the general case of $L^p$  based norms, generalizing
also in this aspect the results of previous papers.
Our arguments are different to those used in previous papers.
The main point is to prove sharp estimates in reference elements.

Let us explain the idea in the two dimensional case. Consider
the reference triangle $\hT$ which has vertices at $(0,0)$, $(1,0)$ and
$(0,1)$. A stability estimate on $\hT$ can be used to obtain
the stability in a general triangle by using the Piola transform.
Afterwards, error estimates can be proved combining stability
with polynomial approximation results.

The original proof given in \cite{RT} uses that
$$
\|\Pi_k\u\|_{L^2(\hT)} \le C \|\u\|_{H^1(\hT)}.
$$
In this way, the constant arising in the estimate for a general
element depends on the minimum angle and so the regularity
assumption is needed. The reason of that dependence is that
the complete $H^1$-norm appears on the right hand side.
Therefore, to improve this result one may try to
obtain sharper estimates on $\hT$ for each component of $\Pi_k\u$.
Denote with $u_j$ and $\Pi_{k,j}\u$, $j=1,2$, the components of
$\u$ and its Raviart-Thomas interpolation respectively and
consider for example $j=1$. Ideally, we would like to have the estimate
$$
\|\Pi_{k,1}\u\|_{L^2(\hT)} \le C \|u_1\|_{H^1(\hT)}.
$$
However, an easy computation shows that if, for example,
$\u=(0,x_2^2)$ then, $\Pi_k\u=\frac13(x_1,x_2)$ and therefore
the above estimate is not true. In other words, even for a rectangular
triangle $\hT$, $\Pi_{k,1}\u$ depends on both components of $\u$.
Now, the question is: which are the essential degrees of freedom
defining $\Pi_{k,1}\u$?

To answer this question one can try to ``kill'' degrees of freedom
by modifying $\u$ without changing $\Pi_{k,1}\u$. A key observation
is that if ${\bf r}=(0,g(x_1))$ then $\Pi_{k,1}{\bf r}=0$ (we will give the proof of this
result for appropriate reference elements in the three dimensional case).
Therefore, if $\vv=(u_1(x_1,x_2),u_2(x_1,x_2)-u_2(x_1,0))$ then,
$\Pi_{k,1}\vv=\Pi_{k,1}\u$. But the normal component of $\vv$ on the
edge $\ell_2$ contained in the line $\{x_2=0\}$, i.e. $v_2$, vanishes,
and so do all the degrees of freedom defining $\Pi_k$ associated with that edge. Moreover,
if we now modify the second component defining
$\w=(u_1(x_1,x_2),u_2(x_1,x_2)-u_2(x_1,0)-x_2\alpha)$, for some
$\alpha\in {\mathcal P}_{k-1}$, we still have that $w_2$ vanishes on $\ell_2$
and that $\Pi_{k,1}\w=\Pi_{k,1}\u$ (because we are modifying $\vv$ by adding
a vector field belonging to the Raviart-Thomas space of order $k$).
But, as we will see, it is possible to choose $\alpha$ in such a way
that the degrees of freedom corresponding to integrals over $\hT$ also vanish.
Of course, it will be necessary to estimate some norm of $\alpha$.
We will give the details of the proofs in the three dimensional
case. It is easy to see that the same arguments can be used to complete
the arguments explained above for the two dimensional case.

The new contributions of this paper can be summarized as follows:

\begin{itemize}

\item We prove error estimates under the maximum angle condition
with order $j+1$ if the approximated function is in $W^{j+1,p}(T)^n$,
$n=2,3$, where $0\le j \le k$ and $1\le p \le\infty$.

\item Under the regular vertex property we obtain estimates
of anisotropic type also for general $k\ge 0$ and $1\le p \le\infty$.
We also show that this kind of estimates is not valid
under the maximum angle condition.

\end{itemize}

Let us finally mention that the interpolation error estimates of
anisotropic type are necessary when one wishes to exploit the
independent element sizes $h_1$, $h_2$ and $h_3$ to treat edge
singularities in elliptic problems or layers in singularly perturbed
problems. The dilemma is that such estimates hold, as we show,
only for tetrahedra with the regular vertex property but it seems to
be impossible to fill space by using this type of elements only. An
anisotropic triangular prism (pentahedron) can, for example, be
subdivided into three tetrahedra, from which only two satisfy the
regular vertex property while the third is of the type of the element
at right hand side of Figure 1. The only known way out so far is
discussed in \cite{FNP}. These authors use pentahedral meshes or
tetrahedral meshes which are obtained by a suitable subdivision of a
pentahedral mesh. Pentahedra based on a regular triangular face
satisfy the regular vertex property by construction. For the
approximation on tetrahedral elements they use a composition of two
interpolation operators in order to avoid the above mentioned
insufficiency with the tetrahedra which do not satisfy the regular
vertex property. This approach is restricted to prismatic domains so far.

The rest of the paper is as follows. In Section 2 we introduce notation
and give some preliminary results on the conditions on tetrahedra that
we will work with.
Then, we prove stability in $\Lp(T)^3$ for the
Raviart-Thomas interpolation of arbitrary degree for functions in
$\W1p(T)^3$. These stability results are proved in Sections 3
for elements satisfying the regular vertex property, and in Section 4
for elements satisfying the maximum angle condition.
The estimates obtained under both hypotheses are essentially different but
the results are sharp. Indeed,
in Section 5 we show that
anisotropic type stability estimates can not
be obtained for the larger class of elements satisfying the
maximum angle condition.
Finally, in Section 6, we derive
the error estimates from the stability results
and standard approximation arguments.

\setcounter{equation}{0}
\section{Notation and Preliminary Results}
In this section we recall some known results involving geometric
properties of certain degenerate tetrahedra. Most of these results
were proved in \cite{AD1} and \cite{Ap}.

Given a general tetrahedron $T\subset \R^3$, ${\bf p}_0$ will
denote an arbitrary vertex
and, for $1\le i \le 3$, $\l_i$, with $\|\l_i\|=1$, will be the directions
of the edges sharing ${\bf p}_0$ and $h_i$ the lengths of those edges.
In other words, $T$ is the convex hull of
$\{{\bf p}_0\}\cup \{{\bf p}_0+h_i\l_i\}_{1\le i\le 3}$.

We will use the standard notation for Sobolev spaces $W^{k,p}(\Omega)$ of
functions with all their derivatives up to the order $k$ belonging
to $L^p(\Omega)$, denoting by $\|\cdot\|_{W^{k,p}(\Omega)}$ the associated norm.
The same notation will be used for the norm of vector fields
${\bf u}\in W^{k,p}(\Omega)^3$. As it is usual, we use boldface
fonts for vector fields.

With $\mathcal P_k(T)$ we denote the set of polynomials of degree
less than or equal $k$ defined over $T\subset \R^3$. The
Raviart-Thomas space of order $k$ is defined as
$$
\mathcal{RT}_k=\mathcal P_k(T)^3 + (x_1,x_2,x_3)\mathcal P_k(T),
$$
and for $\u\in W^{1,p}(T)^3$ the Raviart-Thomas
interpolation of order $k$ is defined as $\Pi_k\u \in\mathcal{RT}_k$
such that
\begin{eqnarray}
\label{RTcaras} \int_F\Pi_k\u\cdot\n p_k & = & \int_F\u\cdot \n
p_k\qquad\forall
p_k\in\mathcal P_k(F),\ \ F\mbox{ face of }T,\\
\label{RTadentro} \int_{T}\Pi_k\u\cdot\p_{k-1} & = &
\int_T\u\cdot\p_{k-1} \qquad \forall\p_{k-1}\in\mathcal
P_{k-1}(T)^3.
\end{eqnarray}

In the rest of the paper the letter $C$ will denote a generic constant that may
change from line to line.

Now we introduce the different conditions on the elements that we will use.
The first one, called {\it ``regular vertex property"} was introduced in \cite{AD1}.

\begin{definition}
A tetrahedron $T$ satisfies the ``regular vertex property" with a
constant $\overline c>0$ (or shortly, \textsl{RVP}$(\bar c)$) if $T$ has a
vertex ${\bf p}_0$, such that if $M$ is the matrix made up with
$\l_i$, $1\le i\le 3$, as columns then $|\det M|>\overline c$.
\end{definition}

One can easily check that a regular family of tetrahedra (with the usual
definition of regularity given for example in \cite{C1}) verifies
the regular vertex property. On the other hand, simple examples like that
at the left hand side of
Figure \ref{prismas} show that arbitrarily narrow elements are allowed in
the class given by \textsl{RVP}$(\bar c)$ for a fixed~$\bar c$.

Despite the
presence of anisotropic elements the regular vertex property
arises as a natural geometric condition if one looks for
Raviart-Thomas interpolation error bounds. Indeed, looking at the vertex
placed at ${\bf p}_0$, one can see that the family of elements satisfying
\textsl{RVP}$(\bar c)$, have three normal vectors (those normals to the
faces sharing ${\bf p}_0$) uniformly linearly independent (see
\cite{AD1} for more details). A reasonable condition, since
the moments of the normal components of vectors fields are used as
degrees of freedom in the Raviart-Thomas interpolation.
Strikingly, as was shown in \cite{AD1}, the uniform independence
of the normal components can be somehow relaxed and error
estimates valid uniformly for a wider class of elements can be still
obtained for $\Pi_0$ (and for $\Pi_k$ as we
will show). More precisely, we will prove error estimates under
the maximum angle condition defined below, which was introduced
by Krizek in \cite{K} and is weaker than the \textsl{RVP}.

\begin{definition}
A tetrahedron $T$ satisfies the ``maximum angle condition" with a
constant $\bar\psi<\pi$ (or shortly \textsl{MAC}$(\bar\psi)$) if the
maximum angle between faces and the maximum angle inside the faces
are less than or equal to $\bar\psi$.
\end{definition}

Let us mention that the estimates obtained under \textsl{RVP} are stronger
than those valid under \textsl{MAC}. Indeed, in the first case the estimates
are of anisotropic type
(roughly speaking, this means that the estimates are given in terms
of sizes in different directions and their corresponding derivatives).
On the other hand, we will show that this kind of estimates are not valid
for the more general class of elements verifying the \textsl{MAC} condition.

The definition of the maximum angle condition,
is strongly geometric. In order to find an equivalent condition, more appropriate
for our further computations, we introduce the following definitions.
In what follows, ${\bf e}_i$ will denote the canonical vectors.

\begin{definition}
A tetrahedron $T$ belongs to the family ${\mathcal F}_1$
if its vertices are at ${\bf 0}$, $h_1{\bf e}_1$,
$h_2{\bf e}_2$ and $h_3{\bf e}_3$,
where $h_i>0$ are arbitrary lengths (see Figure 1a).
\end{definition}

\begin{definition}
A tetrahedron $T$ belongs to the family ${\mathcal F}_2$
if its vertices are at ${\bf 0}$, $h_1{\bf e}_1+ h_2 {\bf e}_2$,
$h_2{\bf e}_2$ and $h_3{\bf e}_3$,
where $h_i>0$ are arbitrary lengths (see Figure 1b).
\end{definition}

Note that elements in $\mathcal F_2$ satisfy \textsl{MAC}$(\frac\pi2)$ but they
do not fulfill \textsl{RVP}$(\bar c)$ for any $\bar c$.

\begin{lemma}
\label{lemmaK}
Let $T$ be a tetrahedron satisfying \textsl{MAC}$(\bar\psi)$. Then we have
\begin{enumerate}
\item If $\alpha\le\beta\le\gamma$ are the angles of an arbitrary
face of $T$, then $\gamma\ge\frac\pi3$ and
$\beta,\gamma\in[(\pi-\bar\psi)/2,\bar\psi]$. \item If ${\bf p}_0$ is an
arbitrary vertex of $T$ and $\chi\le \psi\le\phi$ are the angles
between faces passing through ${\bf p}_0$, then $\phi\ge\frac\pi3$ and
$\psi,\phi\in[(\pi-\bar\psi)/2,\bar\psi]$.
\end{enumerate}
\end{lemma}
\begin{proof} See \cite{K}.
\end{proof}

For a matrix $M\in\R^{3\times 3}$, $\|M\|$ will denote its infinity norm.
The arguments used in the following theorem are essentially contained in the
proof of Theorem 7 in \cite[page 516]{K}. We include some details here
for the sake of clarity.

\begin{theorem}
\label{mac} Let $T$ be a tetrahedron satisfying \textsl{MAC}$(\bar\psi)$.
Then there exists an element $\widetilde T\in\mathcal F_1\cup
\mathcal F_2$ that can be mapped onto $T$ through an affine
transformation $F(\widetilde\x)=M\widetilde\x + {\bf c}$ with
$\|M\|,\|M^{-1}\|\le C$ where the constant $C$ depends only on
$\bar\psi$.
\end{theorem}

\begin{proof} Given a tetrahedron $T$ we denote with ${\bf p}_i$, $i=0,1,2,3$,
its vertices and use obvious notations for its faces and edges.
Let ${\bf p}_0{\bf p}_1{\bf p}_2$ be an arbitrary face of $T$ and
${\bf p}_3$ its opposite vertex. We can assume that the maximum angle
$\gamma$ of the face ${\bf p}_0{\bf p}_1{\bf p}_2$ is at the vertex ${\bf p}_0$. Then from
Lemma \ref{lemmaK} we have
\[
\sin\gamma\ge
m:=\min\left\{\sin\frac{\pi-\bar\psi}{2},\sin\bar\psi\right\}.
\]
Let ${\bf t_1}$ and ${\bf t_2}$ be unit vectors along the edges ${\bf p}_0{\bf p}_1$
and ${\bf p}_0{\bf p}_2$.  We can also assume that the angle $\omega$ between the
faces ${\bf p}_0{\bf p}_1{\bf p}_2$ and ${\bf p}_0{\bf p}_1{\bf p}_3$ is not less than the angle
between ${\bf p}_0{\bf p}_1{\bf p}_2$ and ${\bf p}_0{\bf p}_2{\bf p}_3$ (otherwise we interchange the
notation between the vertices ${\bf p}_1$ and ${\bf p}_2$). Then, again from Lemma
\ref{lemmaK} we have
\[
\sin\omega\ge m.
\]
Now consider the triangle ${\bf p}_0{\bf p}_1{\bf p}_3$ and choose $k\in\{0,1\}$ so
that the angle $\xi$ at the vertex ${\bf p}_k$ is not less than that at the
vertex ${\bf p}_{1-k}$. Using again Lemma \ref{lemmaK} we obtain
\[
\sin\xi\ge m.
\]
We take now ${\bf t_3}$ as the unit vector along ${\bf p}_k{\bf p}_3$ and define
$M_0$ as the matrix made up with ${\bf t_1}, {\bf t_2}$ and ${\bf t_3}$
as its columns.
Since the columns of $M_0$ are unitary vectors we have $\|M_0\|\le 3$.
Also, the adjugate matrix of $M_0$ has coefficients with absolute value
bounded by 2 and therefore, $\|M_0^{-1}\|\le 6/|\det M_0|$.
Then, to obtain the desired bound for $\|M_0^{-1}\|$ it is enough to
show that $|\det M_0|$ is bounded by below by a constant which depends only
on $\bar \psi$.

Consider the parallelepiped generated by the vectors ${\bf t_1}, {\bf t_2}$
and ${\bf t_3}$. Let $z$ be its height in the direction
perpendicular to ${\bf t_1}$ and ${\bf t_2}$ and
$y$ the height of the face generated by ${\bf t_1}$ and ${\bf t_3}$
in the direction perpendicular to ${\bf t_1}$.

Since
$\|{\bf t_i}\|=1$ we have
$$
|\det M_0|=z\sin\gamma=y\sin\omega\sin\gamma=\sin\xi\sin\omega\sin\gamma
\ge m^3.
$$
as we wanted to prove.

Obviously, the same properties are satisfied by the
matrix $M_1$ made up with ${\bf t_2}, -{\bf t_1}$ and ${\bf t_3}$,
as its columns.

Now, define $h_1=|{\bf p}_0{\bf p}_1|$, $h_2=|{\bf p}_0{\bf p}_2|$ and $h_3=|{\bf p}_k{\bf p}_3|$.
If $k=0$ take $\widetilde T \in \mathcal F_1$ with
vertices at ${\bf 0},h_1{\bf e}_1, h_2{\bf e}_2$ and  $h_3{\bf e}_3$
and if $k=1$ take
$\widetilde T \in \mathcal F_2$ with
vertices at ${\bf 0},h_1{\bf e}_1+ h_2 {\bf e}_2,
h_2{\bf e}_2$ and $h_3{\bf e}_3$. Then it is easy to
check that $\widetilde\x\mapsto M_k\widetilde\x+ {\bf p}_k$ maps
$\widetilde T$ onto $T$.
\end{proof}

As mentioned above, the regular vertex property is stronger
than the maximum angle condition. Indeed, the following
theorem shows that, under \textsl{RVP}$(\bar c)$, the reference
family in the previous theorem can be restricted to $\mathcal F_1$.

\begin{theorem}
\label{sobrervp} Let $T$ be a tetrahedron satisfying
\textsl{RVP}$(\bar c)$. Then, there exists an element $\widetilde
T\in\mathcal F_1$ that can be mapped onto $T$ through an affine
transformation $F(\widetilde{ {\bf x}})=M \widetilde{\bf{x}}+{\bf
p}_0$ with $\|M\|,\|M^{-1}\|\le C$ where the constant $C$ depends
only on $\bar c$. Furthermore, if $h_i, i=1,2,3$ are the lengths
of the edges of $T$ sharing the vertex ${\bf p}_0$, we can take
$\widetilde{T}\in {\mathcal F}_1$ such that, for $i=1,2,3$, $h_i$
is the length in the direction ${\bf e}_i$.
\end{theorem}

\begin{proof} Let ${\bf p}_0$ and ${\l_i}$ be as in the
definition of \textsl{RVP}$(\bar c)$ and $h_i$ be the length of the
edge of $T$ with direction $\l_i$.

Take $M$ as the matrix made up with $\l_i$
as its columns. Since $|\det(M)|>\bar c$ and $\l_i$ are unitary vectors
then it is easy to check that
$\|M\|\le C$ and $\|M^{-1}\|\le C$ with a constant $C$ depending only
on a lower bound of $|\det(M)|$ and therefore on $\bar c$.

Then, if $\widetilde T$ is the tetrahedron of $\mathcal F_1$ with
with lengths $h_i$ in the directions ${\bf e}_i$, the affine transformation $F(\widetilde
\x)=M\widetilde\x+\bf p_0$ maps $\widetilde T$ onto $T$.
\end{proof}

\begin{remark} It is not difficult to see that the converses of
  Theorems \ref{mac} and \ref{sobrervp} hold true. Namely, the family
  of elements obtained by transforming $\mathcal F_1\cup\mathcal F_2$
  (resp. $\mathcal F_1)$)
by affine maps $\widetilde\x\mapsto M\widetilde\x+ {\bf c}$, where
$\|M\|,\|M^{-1}\|\le C$, satisfies \textsl{MAC}$(\bar\psi)$ (resp. \textsl{RVP}$(\bar c)$)
for some $\bar\psi$ (resp. $\bar c$) which depends only on $C$.
\end{remark}

\setcounter{equation}{0}
\section{Stability under the regular vertex property}

The goal of this section is to prove the
stability in $L^p$ for the Raviart-Thomas interpolation  of
arbitrary order of functions in $W^{1,p}(T)^3$, for families
of elements satisfying the regular vertex property.
Precisely, the main result of this section is the following theorem.

\begin{theorem}
\label{mainrvp}
Let $k\ge0$ and $T$ be a tetrahedron satisfying \textsl{RVP}$(\bar c)$. If
${\bf p}_0$ is the regular vertex, $\l_i, i=1,2,3$ are unitary
vectors with the directions of the edges sharing ${\bf p}_0$,
$h_i, i=1,2,3$, the lengths of these edges, and $h_T$ the diameter of $T$ then,
there exists a constant $C$ depending only on $k$ and $\bar c$ such that,
for all $\u\in\W1p(T)^3$, $1\le p\le \infty$,
$$
\left\|\Pi_k\u\right\|_{\Lp(T)}\le C\left( \|\u\|_{\Lp(T)} +
\sum_{i,j} h_j\left\|\frac{\partial u_i}{\partial
\ell_j}\right\|_{\Lp(T)} + h_T\|\d \u\|_{\Lp(T)}\right).
$$
\end{theorem}
\bigskip

The theorem will follow by Theorem \ref{sobrervp} once we have proved
error estimates for elements in the family $\mathcal F_1$.

First we will prove appropriate estimates in the reference element
$\widehat T$ defined as the tetrahedron with vertices at
$(0,0,0), (1,0,0), (0,1,0)$ and $(0,0,1)$.
This is the object of the next two lemmas. Afterwards, estimates
for elements in $\mathcal F_1$ will be obtained by scaling arguments.

We denote with $\widehat F_i$ the face of
$\widehat T$ normal to $\n_i$, with $\n_1=(-1,0,0),
\n_2=(0,-1,0), \n_3=(0,0,-1)$ and $\n_4=\frac1{\sqrt{3}}(1,1,1)$.
We will use the same notation for a function
of two variables than for its extension to
$\widehat T$ as a function independent of the other variable,
for example, $f(x_2,x_3)$ will denote a function defined on
$\widehat F_1$ as well as one defined in $\widehat T$ (anyway,
the meaning in each case will be clear from the context).
In the same way, the same notation will be used to denote a polynomial
$p_k$ on a face and a polynomial in three variables such that
its restriction to that face agrees with $p_k$. For example,
for $p_k\in\mathcal{P}_k(\widehat F_4)$ we will write $p_{k}(1-x_2-x_3,x_2,x_3)$.
In what follows $\widehat\Pi_{k,i}\u$ denotes the $i$-th component
of $\widehat\Pi_k\u$.

\begin{lemma}
\label{lemma1}
Let $f\in L^p(\widehat F_1)$, $g\in L^p(\widehat F_2)$,
and $h\in L^p(\widehat F_3)$. If
$$
\u(x_1,x_2,x_3)=(f(x_2,x_3),0,0),
\quad\vv(x_1,x_2,x_3)=(0,g(x_1,x_3),0),
$$
and
$$
\w(x_1,x_2,x_3)=(0,0,h(x_1,x_2))
$$
then, their Raviart-Thomas interpolations are of the same form,
namely, there exist $q_i\in\mathcal{P}_k(\widehat F_i)$, $i=1,2,3$,
such that
$$
\widehat\Pi_k\u=(q_1(x_2,x_3),0,0),\quad
\widehat\Pi_k\vv=(0,q_2(x_1,x_3),0),
$$
and
$$
\widehat\Pi_k{\bf w}=(0,0,q_3(x_1,x_2)).
$$
\end{lemma}
\begin{proof} Let us prove for example the first equality, the other two are
obviously analogous.
Since $\d\u=0$, we have that $\d\widehat\Pi_k\u=0$ and therefore,
from a well known property of the Raviart-Thomas
interpolation (see for example \cite{BF,D3}), it follows that
$\widehat\Pi_k\u\in \mathcal{P}_k(\widehat T)^3$.

On the other hand, using now (\ref{RTcaras}) for $i=2,3$, and that
$u_2=u_3=0$, we have
$$
\int_{\widehat F_i}\widehat\Pi_{k,i}\u\, p_k = 0
\qquad\forall p_k\in\mathcal P_k(F_i),
\quad i=2,3,
$$
and then, taking $p_k=\widehat\Pi_{k,i}\u$, we conclude that
$\widehat\Pi_{k,i}\u|_{\widehat F_i}=0$ for $i=2,3$. Therefore
$\widehat\Pi_{k,i}\u=x_i r_i$ for some $r_i\in\mathcal{P}_{k-1}(\widehat T)$
and so, using now (\ref{RTadentro}) and again that $u_2=u_3=0$,
we obtain that, for $i=2,3$, $\widehat\Pi_{k,i}\u=0$ in $\widehat T$  as we wanted
to show.

Finally, since $\d\widehat\Pi_k\u=0$, it follows that
$\frac{\partial\widehat\Pi_{k,1}\u}{\partial x_1}=0$ and so,
$\widehat\Pi_{k,1}\u$ is independent of $x_1$.
\end{proof}

\begin{lemma}
There exists a constant $C$ depending only on $k$ such that, for all
$\u=(u_1,u_2,u_3)\in \W1p(\hT)^3$,
\begin{eqnarray}\label{ineq1}
\|\widehat\Pi_{k,i}\u\|_{\Lp(\hT)}&\le& C\left(\|u_i\|_{\W1p(\hT)}
+ \|\d\u\|_{\Lp(\hT)}\right), \qquad i=1,2,3.
\end{eqnarray}
\end{lemma}
\begin{proof}
From the previous lemma we know that, if
$$
\vv=(u_1,u_2-u_2(x_1,0,x_3), u_3-u_3(x_1,x_2,0))
$$
then,
$\widehat\Pi_{k,1}\u = \widehat\Pi_{k,1}\vv$.

Let $\alpha,\beta\in\mathcal P_{k-1}(\widehat T)$ be
such that
\begin{equation}
\label{eq10}
\int_{\widehat T}(v_2-x_2\alpha)\,p_{k-1}=0\quad
\mbox{and}\quad
\int_{\widehat T}(v_3-x_3\beta)\,p_{k-1} =0\qquad
\forall p_{k-1}\in\mathcal{P}_{k-1}(\widehat T).
\end{equation}
Observe that those $\alpha$ and $\beta$ exist. Indeed, it is easy to
prove uniqueness (and therefore existence) of solution of the square
linear systems of equations defining them.

Define now $\w=(v_1,v_2-x_2\alpha,v_3-x_3\beta)$.
Since $(0,x_2\alpha,x_3\beta)\in\RT_k(\widehat T)$ it follows that
$\widehat\Pi_{k,1}\vv=\widehat\Pi_{k,1}\w$ and therefore
$\widehat\Pi_{k,1}\u=\widehat\Pi_{k,1}\w$.

Taking into account that $w_2|_{\widehat F_2}=0$ and
$w_3|_{\widehat F_3}=0$ and the equations (\ref{eq10}), it follows
that $\widehat\Pi\w$ is determined by the equations
\begin{eqnarray}\nonumber
\int_{\widehat T}\widehat\Pi_{k,1}\w\, p_{k-1} & = &
\int_{\widehat T}w_1 \,p_{k-1} \qquad \forall p_{k-1}\in\mathcal
P_{k-1}(\widehat
T)\\
\nonumber \int_{\widehat T}\widehat\Pi_{k,2}\w\, p_{k-1} & = & 0
\qquad \forall p_{k-1}\in\mathcal P_{k-1}(\widehat
T)\\
\label{eq13} \int_{\widehat T}\widehat\Pi_{k,3}\w\, p_{k-1} & = &
0 \qquad \forall p_{k-1}\in\mathcal P_{k-1}(\widehat T)\\\nonumber
\int_{\widehat F_1} \widehat\Pi_{k,1}\w\, p_k &=& \int_{\widehat
F_1}w_1\,p_k \qquad \forall p_{k}\in\mathcal P_{k}(\widehat
F_1)\\\nonumber \int_{\widehat F_2} \widehat\Pi_{k,2}\w\, p_k &=&0
\qquad \forall p_{k}\in\mathcal P_{k}(\widehat F_2)\\\nonumber
\int_{\widehat F_3} \widehat\Pi_{k,3}\w\, p_k &=& 0 \qquad \forall
p_{k}\in\mathcal P_{k}(\widehat F_3)\\\nonumber \int_{\widehat
F_4} (\widehat\Pi_{k,1}\w + \widehat \Pi_{k,2}\w + \widehat
\Pi_{k,3}\w)\, p_k &=& \int_{\widehat F_4}(w_1+w_2+w_3)\,p_k
\qquad \forall p_{k}\in\mathcal P_{k}(\widehat F_4).
\end{eqnarray}

Now, for $p_k\in\mathcal P_{k}(\widehat T)$, we have

\begin{eqnarray*}
\int_{\widehat T}\d  \w p_k &=&-\int_{\widehat T}\w\cdot \nabla
p_k + \int_{\partial\widehat T}\w\cdot \n\,p_k\\
&=&-\int_{\widehat T}\w\cdot \nabla p_k +
\frac1{\sqrt{3}}\int_{\widehat F_4}(w_1+w_2+w_3)p_k +
\int_{\partial\widehat T\setminus\widehat F_4}\w\cdot \n\,p_k
\end{eqnarray*}
but, from the definition of $\w$, we have
\[
-\int_{\widehat T}\w\cdot \nabla p_k=-\int_{\widehat
T}w_1\,\frac{\partial p_k}{\partial x_1}\qquad\mbox{and}\qquad
\int_{\partial\widehat T\setminus\widehat F_4}\w\cdot \n\,p_k =
-\int_{\widehat F_1}w_1\,p_k,
\]
therefore, for all $p_{k}\in\mathcal P_{k}(\widehat T)$,
\begin{equation}\label{eq12}
\frac1{\sqrt{3}}\int_{\widehat F_4}(w_1+w_2+w_3)p_k =
\int_{\widehat T}\d \w\,p_{k} + \int_{\widehat
T}w_1\,\frac{\partial p_k}{\partial x_1} + \int_{\widehat
F_1}w_1\,p_k.
\end{equation}
But,
\[
\d \w=\d\vv-\d (0,x_2\alpha,x_3\beta)
= \d\u-\d (0,x_2\alpha,x_3\beta).
\]
So, using (\ref{eq12}), (\ref{eq13}), and standard arguments, we
obtain
\begin{eqnarray*}
\lefteqn{\|\widehat\Pi_{k,1}\u\|_{\Lp(\hT)} =
\|\widehat\Pi_{k,1}\w\|_{\Lp(\hT)}}\\&\le& C\left( \|u_1\|_{\W1p(\hT)}
+ \|\d \u\|_{\Lp(\hT)} + \|\d (0,x_2\alpha,x_3\beta)\|_{\Lp(\hT)}\right).
\end{eqnarray*}
Then, to obtain (\ref{ineq1}) for $i=1$, it is enough to show that
\begin{equation}
\label{ultima cota}
\|\d (0,x_2\alpha,x_3\beta)\|_{\Lp(\hT)} \le
C(\|u_1\|_{\W1p(\hT)} + \|\d \u\|_{\Lp(\hT)}).
\end{equation}

For $p_k\in\mathcal P_{k}(\hT)$ we have
\begin{eqnarray*}
0&=&\int_{\widehat T} (0,v_2-x_2\alpha,v_3-x_3\beta)\cdot\nabla
p_{k}\\
&=&-\int_{\widehat T} \d (0,v_2-x_2\alpha,v_3-x_3\beta)\,p_k +
\int_{\partial\widehat
T}\left[(v_2-x_2\alpha)n_2+(v_3-x_3\beta)n_3\right]\,p_k.
\end{eqnarray*}
Now we take $p_k(x_1,x_2,x_3)=(1-x_1-x_2-x_3)p_{k-1}$ with
$p_{k-1}\in\mathcal P_{k-1}(\hT)$. Then, since $p_k=0$ on
$\widehat F_4$, $(v_2-x_2\alpha)n_2=0$ on $\partial\widehat T\setminus F_4$
and $(v_3-x_3\beta)n_3=0$ on $\partial\widehat T\setminus F_4$, it follows
that, in the last equation, the boundary integral vanishes. Then,
\[
\int_{\widehat T} (1-x_1-x_2-x_3)\,\d
(0,v_2-x_2\alpha,v_3-x_3\beta)\,p_{k-1} =0.
\]
That is, for all $p_{k-1}\in\mathcal P_{k-1}(\widehat T)$,
$$
\int_{\widehat T} (1-x_1-x_2-x_3)\,\d(0,x_2\alpha,x_3\beta)\,p_{k-1}=\int_{\widehat T}
(1-x_1-x_2-x_3)\,\d (0,v_2,v_3)\,p_{k-1}.
$$
Therefore, taking $p_{k-1}=\d(0,x_2\alpha,x_3\beta)$ and applying the H\"older
inequality we obtain
$$
\int_{\widehat T} (1-x_1-x_2-x_3)\,|\d(0,x_2\alpha,x_3\beta)|^2
\le C\|\d (0,v_2,v_3)\|_{L^p(\widehat T)}\,\|\d(0,x_2\alpha,x_3\beta)\|_{L^{p'}(\widehat T)}.
$$
But, since all the norms on $\mathcal P_{k-1}(\widehat T)$ are equivalent
we conclude that,
\begin{equation}
\label{casi esta}
\|\d(0,x_2\alpha,x_3\beta)\|_{L^p(\widehat T)}
\le C\|\d (0,v_2,v_3)\|_{L^p(\widehat T)}.
\end{equation}
Now, observe that $\d (0,v_2,v_3)=\d (0,u_2,u_3)$ and
$$
\|\d (0,u_2,u_3)\|_{\Lp(\hT)}
\le C\left(\|\d \u\|_{\Lp(\hT)} +
\|u_1\|_{\W1p(\hT)}\right)
$$
then, (\ref{ultima cota}) follows from (\ref{casi esta}).

Clearly, the estimates for $\widehat\Pi_{k,2}\u$ and
$\widehat\Pi_{k,3}\u$ can be proved analogously.
\end{proof}

From the previous lemma and a change of variables we obtain
estimates for elements in $\mathcal F_1$.

The Raviart-Thomas operators on the elements $\hT$ and $\widetilde T$ will
be denoted by $\widehat\Pi_k$ and $\widetilde\Pi_k$ respectively.
Analogous notations will be used for variables and derivatives
or differential operators on $\hT$ and $\widetilde T$ whenever
needed for clarity.

\begin{proposition}
\label{proprvp} Let $\tT\in\mathcal F_1$ be the element with vertices
at ${\bf 0}$, $h_1{\bf e}_1$,
$h_2{\bf e}_2$ and $h_3{\bf e}_3$, where $h_i>0$.
There exists a constant $C$ depending only on $k$
such that, for $\widetilde\u=(\tilde u_1,\tilde u_2,\tilde u_3)\in \W1p(\widetilde T)^3$
and $i=1,2,3$,
\begin{eqnarray*}
\label{ineq0}
\|\widetilde\Pi_{k,i}\widetilde\u\|_{\Lp(\widetilde T)}&\le&
C\left(\|\tilde u_i\|_{\Lp(\widetilde T)} + \sum_{j=1}^3 h_j
\left\| \frac{\partial \tilde u_i}{\partial \tilde x_j}
\right\|_{\Lp(\widetilde T)} +
h_i\|\widetilde\d\widetilde\u\|_{\Lp(\widetilde T)}\right)
\end{eqnarray*}
\end{proposition}
\begin{proof} Let $\widehat\u\in W^{1,p}(\hT)^3$ defined via the
Piola transform by
$$
\widetilde\u(\widetilde\x)=
\frac1{\det B}B\widehat\u(\widehat\x),\quad\widetilde\x=B\widehat\x,
\quad \mbox{with}\ \
B=\left(%
\begin{array}{ccc}
  h_1 & 0 & 0 \\
  0 & h_2 & 0 \\
  0 & 0 & h_3 \\
\end{array}%
\right)
$$
It is known that (see for example \cite{D3,RT}),
\begin{equation}
\label{Piola1}
\widetilde\Pi_k\widetilde\u(\widetilde\x)=
\frac1{\det B}\,B\,\widehat\Pi_k\widehat\u(\widehat\x),
\end{equation}
and
\begin{equation}
\label{Piola2}
\widetilde\d\widetilde\u(\widetilde\x) =
\frac1{\det B}\widehat\d\widehat\u(\widehat\x).
\end{equation}
Consider for example $i=1$ (the other cases are of course analogous).
Using (\ref{ineq1}) we have
\begin{eqnarray*}
\left\|\widetilde\Pi_{k,1}\widetilde\u\right\|_{L^p(\tT)}^p &=&
\frac{h_1h_2h_3}{h_2^ph_3^p}\left\|\widehat\Pi_{k,1}\widehat\u\right\|_{L^p(\hT)}^p\\
&\le& C\frac{h_1h_2h_3}{h_2^ph_3^p} \left(\|\hat
u_1\|_{W^{1,p}(\hT)}^p +
\|\widehat\d\widehat\u\|_{L^p(\hT)}^p\right)\\
&\le& C\left(\|\tilde u_1\|_{\Lp(\widetilde T)}^p + \sum_{j=1}^3
h_j^p \left\| \frac{\partial \tilde u_1}{\partial \tilde x_j}
\right\|_{\Lp(\widetilde T)}^p +
h_1^p\|\widetilde\d\widetilde\u\|_{\Lp(\widetilde T)}^p\right)
\end{eqnarray*}
as we wanted to show.
\end{proof}

We are finally ready to prove the main theorem of
this section.

\begin{proof}[Proof of Theorem \ref{mainrvp}]
To simplify notation we assume ${\bf p}_0={\bf 0}$. From Theorem
\ref{sobrervp} we know that, if $\tT\in\mathcal F_1$ is the
element with vertices at ${\bf 0}$, $h_1{\bf e}_1$, $h_2{\bf e}_2$
and $h_3{\bf e}_3$, there exists a matrix $M$ such that the
associated linear transformation maps $\tT$ onto $T$. Moreover,
$M{\bf e_i}=\l_i, i=1,2,3$.

Now, given $\u\in\W1p(T)^3$ we define $\widetilde\u\in\W1p(\widetilde T)^3$
via the Piola transform, namely,
\[
\u(\x)=\frac1{\det M}M\widetilde\u(\widetilde\x),\qquad
\x=M\widetilde\x\, .
\]
Using Proposition \ref{proprvp} after the change of variables
$\x\mapsto\widetilde\x$ we have
\[
\|\Pi_k\u\|_{L^p(T)}^p \le C \frac{\|M\|^p}{(\det M)^{p-1}}
\left(
\|\widetilde\u\|_{L^p(\tT)}^p + \sum_{j=1}^3 h_j^p \left\|
\frac{\partial\widetilde\u}{\partial\tilde x_j}\right\|_{L^p(\tT)}^p +
h_{\tT}^p\|\widetilde\d\widetilde\u\|_{L^p(\tT)}^p\right)
\]
where $h_{\tT}$ is the diameter of $\tT$ and
$\frac{\partial\widetilde\u}{\partial\tilde x_j}$ denotes the
vector $\left(\frac{\partial\tilde u_1}{\partial\tilde x_j},
\frac{\partial\tilde u_2}{\partial\tilde x_j},
\frac{\partial\tilde u_3}{\partial\tilde x_j}\right)^t$. But,
\begin{equation}
\label{propiedadpiola}
\frac{\partial\widetilde\u}{\partial\tilde x_j} = \det M\, M^{-1}
\frac{\partial\u}{\partial \l_j},\qquad
\d\u(\x)=\frac1{\det M}\widetilde\d\widetilde\u(\widetilde\x),
\end{equation}
and $h_{\tT}\le \|M^{-1}\| h_T$. Therefore we arrive at
\[
\|\Pi_k\u\|_{L^p(T)}^p \le C \|M\|^p\|M^{-1}\|^p \left(
\|\u\|_{L^p(T)}^p + \sum_{j=1}^3 h_j^p \left\|
\frac{\partial\u}{\partial \l_j}\right\|_{L^p(T)}^p + h_T^p\|\d\u\|_{L^p(T)}^p\right)
\]
and recalling that $\|M\|,\|M^{-1}\|\le C$ with $C$ depending
only on $\bar c$, we conclude the proof. \end{proof}

\setcounter{equation}{0}
\section{Stability under the maximum angle condition}
\label{section4}

In this section we prove a stability result weaker
than that obtained in the previous section but which is valid
for families of elements satisfying the maximum angle condition.

The estimate obtained here, although uniform in the class of
elements satisfying \textsl{MAC}$(\bar\psi)$, is weaker than the estimate
obtained in Theorem \ref{mainrvp} under the stronger \textsl{RVP}$(\bar c)$ hypothesis.
Indeed, in front of each derivative, it appears the diameter $h_T$
instead of the length of the edge in the direction
of the derivative. However, our result is optimal. In fact,
we will show in the next section that estimates like those in
Theorem \ref{mainrvp} are not valid in general under the maximum
angle condition.

The main result of this section is the following theorem.

\begin{theorem}
\label{mainmac} Let $k\ge0$ and $T$ be a tetrahedron with diameter $h_T$
satisfying \textsl{MAC}$(\bar\psi)$. There exists a constant $C$
depending only on $k$ and $\bar\psi$ such that, for all $\u\in \W1p(T)^3$,
$1\le p\le \infty$,
\begin{equation}
\label{mainmac1}
\|\Pi_k\u\|_{\Lp(T)}\le C\left(\|\u\|_{\Lp(T)}
+ h_T \|\nabla\u\|_{\Lp(T)}\right).
\end{equation}
\end{theorem}

\bigskip

The steps to prove this theorem are similar to those followed in
Section 3. Now our reference element $\widehat T$ is the
tetrahedron with vertices at ${\bf 0}$, ${\bf e}_1+ {\bf e}_2$,
${\bf e}_2$ and ${\bf e}_3$. For $\n_1=(1,0,0),
\n_2=\frac1{\sqrt{2}}(1,-1,0), \n_3=(0,0,1)$ and
$\n_4=\frac1{\sqrt{2}}(0,1,1)$ we denote with $\widehat F_i$ the
face of $\widehat T$ normal to $\n_i$ and with $\overline{F}_2$
the projection of $\widehat F_2$ onto the plane given by $x_2=0$.

\begin{lemma}
\label{lemma31} Let $f\in L^p(\widehat F_1)$, $g\in L^p(\overline
F_2)$, and $h\in L^p(\widehat F_3)$. If
$$
\u(x_1,x_2,x_3)=(f(x_2,x_3),0,0),
\quad\vv(x_1,x_2,x_3)=(0,g(x_1,x_3),0),
$$
and
$$
\w(x_1,x_2,x_3)=(0,0,h(x_1,x_2))
$$
then, their Raviart-Thomas interpolations are of the same form,
namely, there exist $q_i\in\mathcal{P}_k(\widehat F_i)$, $i=1,3$,
and $q_2\in\mathcal{P}_k(\overline F_2)$ such that
$$
\widehat\Pi_k\u=(q_1(x_2,x_3),0,0),\quad
\widehat\Pi_k\vv=(0,q_2(x_1,x_3),0),
$$
and
$$
\widehat\Pi_k{\bf w}=(0,0,q_3(x_1,x_2)).
$$
\end{lemma}
\begin{proof} The proof is similar to that of Lemma \ref{lemma1}.
We will prove the first equality, the other two follow in
an analogous way.

First, we have that $\d\widehat\Pi_k\u=0$ and therefore
$\widehat\Pi_k\u\in \mathcal{P}_k(\widehat T)^3$.
Then, proceeding exactly as in the proof of Lemma \ref{lemma1},
we obtain that $\widehat\Pi_{k,3}\u=0$ in $\widehat T$.
Analogously, using now (\ref{RTcaras}) for $i=4$ we have
$(\widehat\Pi_{k,2}\u + \widehat\Pi_{k,3}\u)|_{\widehat F_4}=0$,
and so
$$
\widehat\Pi_{k,2}\u + \widehat\Pi_{k,3}\u=(1-x_2-x_3)r
$$
for some $r\in\mathcal{P}_{k-1}(\widehat T)$. Consequently, using
now (\ref{RTadentro}) and that $u_2=u_3=0$,
we obtain $\widehat\Pi_{k,2}\u + \widehat\Pi_{k,3}\u=0$ in $\widehat T$.
Then, since we already know that $\widehat\Pi_{k,3}\u=0$,
we conclude that  $\widehat\Pi_{k,2}\u=0$ in $\widehat T$.

Therefore, $\widehat\Pi_k\u=(q,0,0)$ for some $q\in\mathcal{P}_k(\widehat T)$
but, since $\d\widehat\Pi_k\u=0$, it follows that $\widehat\Pi_{k,1}$ is independent
of $x_1$.
\end{proof}

\begin{lemma}
There exists a constant $C_1$ depending only on $k$ such that, for all
$\u=(u_1,u_2,u_3)\in \W1p(\hT)^3$,
\begin{eqnarray}\label{ineq1.5}
\|\widehat\Pi_{k,1}\u\|_{L^2(\hT)}&\le&
C_1\left(\|u_1\|_{\W1p(\hT)} + \|\d\u\|_{\Lp(\hT)}\right)\\
\label{ineq2} \|\widehat\Pi_{k,2}\u\|_{L^2(\hT)}&\le&
C_1\left(\|u_2\|_{\W1p(\hT)} + \left\|\frac{\partial u_1}{\partial x_1}\right\|_{\Lp(\hT)}
+ \left\|\frac{\partial u_3}{\partial
x_3}\right\|_{\Lp(\hT)}\right)\\
\label{ineq3}\|\widehat\Pi_{k,3}\u\|_{\Lp(\hT)}&\le&
C_1\left(\|u_3\|_{\W1p(\hT)} + \|\d\u\|_{\Lp(\hT)}\right).
\end{eqnarray}
In particular, for $i=1,2,3$,
\begin{equation}
\label{ineq3.5}
\|\widehat\Pi_{k,i}\u\|_{L^2(\hT)}
\le C_2\left(\|u_i\|_{\W1p(\hT)}
+ \sum_{j=1\atop j\neq i}^3\left\|\frac{\partial u_j}{\partial x_j}\right\|_{\Lp(\hT)}\right)
\end{equation}
for another constant $C_2$ which depends only on $k$.
\end{lemma}
\begin{proof}
Let
$\vv=(u_1,u_2-u_2(x_1,x_1,x_3),u_3-u_3(x_1,x_2,0))$ and
$\alpha,\beta\in \mathcal P_{k-1}(\widehat T)$ such that
\begin{equation}
\label{int0}
\int_{\widehat T}(v_2-(x_1-x_2)\alpha)p_{k-1} =0\quad
\mbox{and}\quad
\int_{\widehat T}(v_3-x_3\beta)p_{k-1} =0\qquad \forall p_{k-1}\in
\mathcal P_{k-1}(\widehat T),
\end{equation}
and define
\[\w=(u_1,u_2-u_2(x_1,x_1,x_3)-(x_1-x_2)\alpha,u_3-u_3(x_1,x_2,0)-x_3\beta).\]
Then, since
$(0,(x_1-x_2)\alpha,x_3\beta)\in\RT_k$, it follows from Lemma \ref{lemma31}
that
$$
\widehat\Pi_{k,1}\u=\widehat\Pi_{k,1}\w.
$$
Now, taking into account the definition of $\w$ and (\ref{int0}), we have that
$\widehat\Pi_k\w$ is defined by
\begin{eqnarray}\nonumber
\int_{\hT}\widehat\Pi_{k,1}\w\, p_{k-1} & = & \int_{\hT}w_1
\,p_{k-1} \qquad \forall p_{k-1}\in\mathcal P_{k-1}(\widehat
T)\\
\nonumber \int_{\widehat T}\widehat\Pi_{k,2}\w\, p_{k-1} & = & 0
\qquad \forall p_{k-1}\in\mathcal P_{k-1}(\widehat
T)\\
\label{eq15} \int_{\widehat T}\widehat\Pi_{k,3}\w\, p_{k-1} & = &
0 \qquad \forall p_{k-1}\in\mathcal P_{k-1}(\hT)\\\nonumber
\int_{\widehat F_1} \widehat\Pi_{k,1}\w\, p_k &=& \int_{\widehat
F_1}w_1\,p_k \qquad \forall p_{k}\in\mathcal P_{k}(\widehat
F_1)\\\nonumber \int_{\widehat F_2}
(\widehat\Pi_{k,1}\w-\widehat\Pi_{k,2}\w)\, p_k &=&\int_{\widehat
F_2}w_1\,p_k \qquad \forall p_{k}\in\mathcal P_{k}(\widehat
F_2)\\\nonumber \int_{\widehat F_3} \widehat\Pi_{k,3}\w\, p_k &=&
0 \qquad \forall p_{k}\in\mathcal P_{k}(\widehat F_3)\\\nonumber
\int_{\widehat F_4} (\widehat\Pi_{k,2}\w + \widehat\Pi_{k,3}\w)\,
p_k &=& \int_{\widehat F_4}(w_2+w_3)\,p_k \qquad \forall
p_{k}\in\mathcal P_{k}(\widehat F_4).
\end{eqnarray}
But, using again (\ref{int0}) we have, for $p_k\in\mathcal P_{k}(\widehat T)$,
\begin{eqnarray}
\int_{\widehat T}\d(0,w_2,w_3)p_k & = & -\int_{\widehat T}
(0,w_2,w_3)\cdot\nabla p_k  \nonumber\\
& + & \int_{\partial\widehat T\setminus\widehat F_4}
(w_2n_2+w_3n_3)p_k + \frac1{\sqrt{2}}\int_{\widehat F_4}(w_2+w_3)p_k\nonumber\\
&=&\frac1{\sqrt{2}}\int_{\widehat F_4}(w_2+w_3)p_k.\label{eq16}
\end{eqnarray}
Then, it follows from (\ref{eq15}) and (\ref{eq16}) that
\begin{eqnarray*}
\|\widehat\Pi_{1,k}\w\|_{\Lp(\hT)} &\le& C\left(
\|w_1\|_{\W1p(\hT)} +
\|\d(0,w_2,w_3)\|_{\Lp(\hT)}\right)\\
&\le & C \left( \|u_1\|_{\W1p(\hT)} + \|\d\u\|_{\Lp(\hT)}
+ \|\d(0,(x_1-x_2)\alpha,x_3\beta)\|_{\Lp(\hT)}\right).
\end{eqnarray*}
Therefore, to conclude the proof of (\ref{ineq1.5}) it is enough to show that
\begin{equation}
\label{eq17}
\|\d(0,(x_1-x_2)\alpha,x_3\beta)\|_{\Lp(\hT)} \le C\,(
\|u_1\|_{\W1p(\hT)} + \|\d\u\|_{\Lp(\hT)}).
\end{equation}
But, for all $p_k\in\mathcal P_{k}(\hT)$, we have
$$
0=\int_{\hT} (0,w_2,w_3)\cdot \nabla p_k
= -\int_{\hT}\d(0,w_2,w_3) + \int_{\partial \hT}
(w_2n_2+w_3n_3)p_k.
$$
Now, taking $p_k=(1-x_2-x_3)p_{k-1}$ with $p_{k-1}\in\mathcal P_{k-1}(\widehat T)$
the boundary integral in the last
equation vanishes, and therefore we obtain
$$
\int_{\hT} (1-x_2-x_3)\d(0,(x_1-x_2)\alpha,x_3\beta)p_{k-1}
=\int_{\hT} (1-x_2-x_3)\d(0,u_2,u_3)p_{k-1}.
$$
Then, (\ref{eq17}) can be obtained with an argument like that used for
(\ref{ultima cota}).
Clearly, the proof of inequality (\ref{ineq3})
is analogous to that of (\ref{ineq1.5}).

Now, to prove (\ref{ineq2}), take
$\vv=(u_1-u_1(0,x_2,x_3),u_2,u_3-u_3(x_1,x_2,0))$,
$\alpha,\beta\in \mathcal P_{k-1}(\widehat T)$ such that
\[
\int_{\hT}(v_1-x_1\alpha)p_{k-1} =0 \quad
\mbox{and}\quad
\int_{\hT}(v_3-x_3\beta)p_{k-1} =0\quad \forall p_{k-1}\in
\mathcal P_{k-1}(\hT),
\]
and define
$$
\w=(v_1-x_1\alpha,v_2,v_3-x_3\beta).
$$
Using again Lemma \ref{lemma31} and that
$(x_1\alpha,0,x_3\beta)\in\RT_k$ we obtain
\[
\widehat\Pi_{k,2}\u=\widehat\Pi_{k,2}\w.
\]
In this case, it follows from the definition of $\w$ that
$\Pi_k\w$ is defined by
\begin{eqnarray}\nonumber
\int_{\widehat T}\widehat\Pi_{k,1}\w\, p_{k-1} & = & 0 \qquad
\forall p_{k-1}\in\mathcal P_{k-1}(\widehat
T)\\
\nonumber \int_{\widehat T}\widehat\Pi_{k,2}\w\, p_{k-1} & = &
\int_{\widehat T} w_2\,p_{k-1} \qquad \forall p_{k-1}\in\mathcal
P_{k-1}(\widehat
T)\\
\label{eq20} \int_{\widehat T}\widehat\Pi_{k,3}\w\, p_{k-1} & = &
0 \qquad \forall p_{k-1}\in\mathcal P_{k-1}(\widehat T)\\\nonumber
\int_{F_1} \widehat\Pi_{k,1}\w\, p_k &=& 0 \qquad \forall
p_{k}\in\mathcal P_{k}(\widehat F_1)\\\nonumber \int_{\widehat
F_2} (\widehat\Pi_{k,1}\w-\widehat\Pi_{k,2}\w)\, p_k
&=&\int_{F_2}(w_1-w_2)\,p_k \qquad \forall p_{k}\in\mathcal
P_{k}(\widehat F_2)\\\nonumber \int_{\widehat F_3} \Pi_{k,3}\w\,
p_k &=& 0 \qquad \forall p_{k}\in\mathcal P_{k}(\widehat
F_3)\\\nonumber \int_{\widehat F_4} (\widehat\Pi_{k,2}\w +
\widehat\Pi_{k,3}\w)\, p_k &=& \int_{\widehat F_4}(w_2+w_3)\,p_k
\qquad \forall p_{k}\in\mathcal P_{k}(\widehat F_4).
\end{eqnarray}
But, it is easy to check by integration by parts that,
for all $p_k\in\mathcal P_{k}(\widehat T)$,
\begin{equation}
\int_{\widehat T}\d(0,w_2,w_3)p_k
= - \int_{\widehat T} w_2\frac{\partial p_k}{\partial x_2}-
\frac1{\sqrt{2}}\int_{\widehat F_2}w_2\,p_k +
\frac1{\sqrt{2}}\int_{\widehat F_4}(w_2+w_3)p_k
\label{eq18}
\end{equation}
and
\begin{equation}
\int_{\widehat T}\d(w_1,w_2,0)p_k
= - \int_{\widehat T} w_2\frac{\partial p_k}{\partial x_2}+
\frac1{\sqrt{2}}\int_{\widehat F_4}w_2\,p_k +
\frac1{\sqrt{2}}\int_{\widehat F_2}(w_1-w_2)p_k\label{eq19}
\end{equation}
Now, it follows from (\ref{eq20}), (\ref{eq18}) and (\ref{eq19}) that
$$
\|\widehat\Pi_{2,k}\w\|_{\Lp(\hT)}
\le C\left( \|w_2\|_{\W1p(\hT)}
+ \|\d(0,w_2,w_3)\|_{\Lp(\hT)} + \|\d(w_1,w_2,0)\|_{\Lp(\hT)}\right)
$$
and therefore, using the definition of $\w$, we obtain
\begin{eqnarray*}
\|\widehat\Pi_{2,k}\w\|_{\Lp(\hT)}
&\le& C \left( \|u_2\|_{\W1p(\hT)}
+ \left\|\frac{\partial u_1}{\partial x_1}\right\|_{\Lp(\hT)}
+\right.\\&&\left.
\left\|\frac{\partial u_3}{\partial x_3}\right\|_{\Lp(\hT)}
+ \left\|\frac{\partial (x_1\alpha)}{\partial x_1}\right\|_{\Lp(\hT)}
+ \left\|\frac{\partial(x_3\beta)}{\partial x_3}\right\|_{\Lp(\hT)}\right).
\end{eqnarray*}
Then, to conclude the proof of (\ref{ineq2}) we have to estimate
the last two terms in the above inequality. From the definition of
$w_3$ we have, for all $p_k\in\mathcal P_k(\hT)$,
$$
0=\int_{\widehat T} w_3\frac{\partial p_k}{\partial x_3}
= - \int_{\widehat T} \frac{\partial w_3}{\partial x_3}p_k
+ \int_{\partial\widehat T} w_3n_3p_k,
$$
but, if we take $p_k=(1-x_2-x_3)p_{k-1}$ with $p_{k-1}\in\mathcal P_{k-1}(\widehat T)$
the boundary integral in the
last equation vanishes, and therefore
$$
\int_{\widehat T} (1-x_2-x_3)\frac{\partial (x_3\beta)}{\partial
x_3}p_{k-1} = \int_{\widehat T} (1-x_2-x_3)\frac{\partial
u_3}{\partial x_3}p_{k-1} \qquad \forall p_{k-1}\in\mathcal
P_{k-1}(\widehat T),
$$
from which we obtain
\[
\left\|\frac{\partial (x_3\beta)}{\partial x_3}\right\|_{\Lp(\hT)}
\le C\,\left\|\frac{\partial u_3}{\partial
x_3}\right\|_{\Lp(\hT)}.
\]
In a similar way we can prove
\[
\left\|\frac{\partial (x_1\alpha)}{\partial
x_1}\right\|_{\Lp(\hT)} \le C\,\left\|\frac{\partial u_1}{\partial
x_1}\right\|_{\Lp(\hT)}
\]
and so (\ref{ineq2}) is proved.\end{proof}

Proceeding as in the previous section we obtain now
estimates for elements in~$\mathcal F_2$.

\begin{proposition}
\label{propmac}
Let $\tT\in\mathcal F_2$ be the element with vertices
at  ${\bf 0}$, $h_1{\bf e}_1+ h_2 {\bf e}_2$,
$h_2{\bf e}_2$ and $h_3{\bf e}_3$, where $h_i>0$.
There exists a constant $C$ depending only on $k$
such that, for $\widetilde\u=(\tilde u_1,\tilde u_2,\tilde u_3)\in \W1p(\widetilde T)^3$
and $i=1,2,3$,

\begin{equation}
\label{ineq4}
\|\widetilde\Pi_{k,i}\widetilde\u\|_{\Lp(\widetilde T)}
\le C\left(\|\tilde u_i\|_{\Lp(\widetilde T)}
+ \sum_{j=1}^3 h_j\left\| \frac{\partial \tilde u_i}{\partial \tilde x_j}
\right\|_{\Lp(\widetilde T)}
+ h_i \sum_{j=1\atop j\neq i}^3  \left\|\frac{\partial \tilde u_j}{\partial \tilde x_j}
\right\|_{\Lp(\widetilde T)}\right)
\end{equation}
\end{proposition}

\begin{proof} We proceed as in the proof of Proposition \ref{proprvp}.
Recall that now our reference element $\widehat T$ is the tetrahedron with vertices
at ${\bf 0}$, ${\bf e}_1+ {\bf e}_2$,
${\bf e}_2$ and ${\bf e}_3$. Therefore, the same linear map given
by $B$ in Proposition \ref{proprvp} maps $\hT$ in $\widetilde T$.
Then, if $\widehat\u\in W^{1,p}(\hT)^3$ is defined via the Piola transform
we have

\begin{equation}
\label{Piola3}
\widetilde\Pi_k\widetilde\u(\widetilde\x)=
\frac1{\det B}\,B\,\widehat\Pi_k\widehat\u(\widehat\x),
\end{equation}
and
\begin{equation}
\label{Piola4}
\widetilde\d\widetilde\u(\widetilde\x) =
\frac1{\det B}\widehat\d\widehat\u(\widehat\x).
\end{equation}
Using (\ref{ineq3.5}) and changing variables we have
\begin{eqnarray*}
\left\|\widetilde\Pi_{k,i}\widetilde\u\right\|_{L^p(\tT)}^p
&=&\frac{h_1h_2h_3}{h_2^ph_3^p}\left\|\widehat\Pi_{k,i}\widehat\u\right\|_{L^p(\hT)}^p\\
&\le& C \frac{h_1h_2h_3}{h_2^ph_3^p}\left(\|\hat
u_i\|^p_{\Lp(\hT)} + \sum_{j=1}^3 \left\| \frac{\partial \hat
u_i}{\partial \hat x_j}\right\|^p_{\Lp(\hT)} + \sum_{j=1\atop
j\neq i}^3  \left\|\frac{\partial \hat u_j}{\partial \hat x_j}
\right\|^p_{\Lp(\hT)}\right)\\
&=&C \left(\|\tilde u_i\|^p_{\Lp(\tilde T)} + \sum_{j=1}^3
h_j^p\left\| \frac{\partial \tilde u_i}{\partial \tilde
x_j}\right\|^p_{\Lp(\widetilde T)} + h_i^p \sum_{j=1\atop j\neq
i}^3   \left\|\frac{\partial \tilde u_j}{\partial \tilde x_j}
\right\|^p_{\Lp(\widetilde T)}\right)
\end{eqnarray*}
and therefore (\ref{ineq4}) is proved.\end{proof}

\begin{remark} For $i=1$ and $i=3$ a better result can be obtained. Indeed,
by the same arguments used in the previous proposition, but using now
(\ref{ineq1.5}) and (\ref{ineq3}), we can prove the following estimates,

$$
\left\|\widetilde\Pi_{k,i}\widetilde\u\right\|_{L^p(\tT)}
\le C\left(\|\tilde u_i\|_{\Lp(\widetilde T)} + \sum_{j=1}^3
h_j \left\| \frac{\partial \tilde u_i}{\partial \tilde x_j}
\right\|_{\Lp(\widetilde T)} +
h_2\|\widetilde\d\widetilde\u\|_{\Lp(\widetilde T)}\right).
$$
Anyway, this clearly depends on the particular orientation
of the element and so, it does not seem to be useful for
general tetrahedra.
\end{remark}

We can now prove the main result of this section.

\begin{proof}[Proof of Theorem \ref{mainmac}] From Theorem \ref{mac}
we know that there exists $\widetilde T\in \mathcal F_1\cup\mathcal F_2$
that can be mapped onto $T$ through an affine transformation
$\widetilde\x\mapsto M\widetilde\x+{\bf c}$, with $\|M\|,\|M^{-1}\|\le C$
for a constant $C$ depending only on $\bar\psi$. To simplify notation assume
that ${\bf c}={\bf 0}$.

If $\tT\in\mathcal F_1$ then, $T$ satisfies the regular vertex property
with a constant which depends only on $\bar\psi$ and so (\ref{mainmac1})
follows immediately from Theorem \ref{mainrvp}. Therefore, we
may assume that $\tT\in\mathcal F_2$ and has vertices at
${\bf 0}$, $h_1{\bf e}_1+ h_2 {\bf e}_2$,
$h_2{\bf e}_2$ and $h_3{\bf e}_3$, where $h_i>0$.

Given $\u\in\W1p(T)^3$ we use again the Piola transform and define
$\widetilde\u\in\W1p(\widetilde T)^3$ given by
\[
\u(\x)=\frac1{\det M}M\widetilde\u(\widetilde\x),\qquad
\x=M\widetilde\x\, .
\]

Then, using that
\[
\Pi_k\u(\x)=\frac1{\det M}M\,\widetilde\Pi_k\widetilde\u(\widetilde\x),
\]
changing variables and using (\ref{mainmac1}) in $\tT$
we obtain
$$
\|\Pi_k\u\|_{\Lp(T)}^p
\le C\|M\|^p\|M^{-1}\|^p \left( \|\u\|_{\Lp(T)}^p + h_T^p
\|M\|^p\|D\u\|_{\Lp(T)}^p\right)
$$
concluding the proof.\end{proof}

\setcounter{equation}{0}
\section{Sharpness of the results}
In view of the results of the previous sections, it is natural to ask whether
the estimate obtained under the maximum angle condition could be improved.
The goal of this section is to show that this is not possible.

Consider the element $\widetilde T\in\mathcal F_2$ with vertices at
${\bf 0}$, $h_1{\bf e}_1+ h_2 {\bf e}_2$,
$h_2{\bf e}_2$ and $h_3{\bf e}_3$ and with diameter $h_T$.
We are going to show that the inequality
\begin{equation}
\label{ineq9}
\|\widetilde\Pi_{k,2}\widetilde\u\|_{\Lp(\widetilde T)}\le
C\left(\|\widetilde \u\|_{\Lp(\widetilde T)} + \sum_{i,j=1}^3 h_j
\left\| \frac{\partial \tilde u_i}{\partial \tilde x_j}
\right\|_{\Lp(\widetilde T)} +
h_T\|\widetilde\d\widetilde\u\|_{\Lp(\widetilde T)}\right),
\end{equation}
with a constant $C$ independent of
$h_1, h_2$ and $h_3$,
does not hold for some
$\widetilde\u=(\tilde u_1,\tilde u_2,\tilde u_3)\in \W1p(\widetilde T)^3$.

Suppose that (\ref{ineq9}), with $C$ independent of $h_1, h_2$ and
$h_3$, holds true for all
$\widetilde\u\in\W1p(\widetilde T)^3$  . Let $\hT$ be the
reference element used in Section \ref{section4}, i.e.,
$\widehat T$ has vertices
at ${\bf 0}$, ${\bf e}_1+ {\bf e}_2$,
${\bf e}_2$ and ${\bf e}_3$. Then, with $\widehat\u\in \W1p(\hT)^3$
we associate $\widetilde \u \in \W1p(\widetilde T)^3$ defined via the
Piola transform with the linear transformation used in the proof of
Theorem \ref{mainmac}.

To simplify notation we drop the hat from now on and
write $\u$ instead of $\widehat\u$ and $x_i$ the variables
in $\widehat T$.

A simple computation shows that from inequality (\ref{ineq9})
we obtain
\begin{eqnarray*}
\|\Pi_{k,2}\u\|_{\Lp(\hT)} & \le &
C\frac1{h_2}\left(\sum_{i=1}^3 h_i \|\u_i\|_{\W1p(\hT)} +
h_T\|\d\u\|_{\Lp(\hT)}\right).
\end{eqnarray*}
Then, taking $h_1=h_3=h_2^2$ (with $h_2<1$), we would have
\begin{eqnarray*}
\|\Pi_{k,2}\u\|_{\Lp(\hT)} & \le & C\left(h_2
\|u_1\|_{\W1p(\hT)}+ \|u_2\|_{\W1p(\hT)} + h_2 \|u_3\|_{\W1p(\hT)}
+ \|\d\u\|_{\Lp(\hT)}\right),
\end{eqnarray*}
and letting $h_2\to 0$ we would arrive at
\begin{equation}\label{ineq8}
\|\Pi_{k,2}\u\|_{\Lp(\hT)}
\le C\left(\|u_2\|_{\W1p(\hT)} + \|\d\u\|_{\Lp(\hT)}\right).
\end{equation}
However, we are going to show that there exists $\u\in \W1p(\hT)^3$
for which inequality (\ref{ineq8}) is not valid.
In fact, in the next proposition we will give, for each
$k\ge 0$, a function $\u\in \W1p(\widehat T)^3$ such that
the right hand side of (\ref{ineq8}) vanishes while the left hand
side does not. We will use the notation of Section \ref{section4}
for the faces of $\widehat T$.

\begin{proposition} For $k\ge 0$,
the function $\u(x_1,x_2,x_3)=(x_1^{k+1},0,-(k+1)x_1^kx_3)$
verifies $\d\u=0, u_2=0$ and $\Pi_{k,2}\u\ne0$.
\end{proposition}

\begin{proof} We consider the case $k\ge1$ (the case $k=0$ follows analogously).
Since $\d\u=0$ we have $\Pi_{k,1}\u, \Pi_{k,3}\u \in\mathcal P_k(\widehat T)$.

Now, using that $u_1=0$ on $\widehat F_1$ and $u_3=0$ on $\widehat F_3$ it
follows from the definition of $\Pi_k\u$ that
$$
\int_{\widehat F_1}\Pi_{k,1}\u\,p_k =0
\qquad \forall p_k\in\mathcal P_k(\widehat F_1)
$$
and
$$
\int_{\widehat F_3}\Pi_{k,3}\u\,p_k =0
\qquad \forall p_k\in\mathcal P_k(\widehat F_3).
$$
Then $\Pi_{k,1}\u=x_1\alpha$ and $\Pi_{k,3}\u=x_3\beta$
with $\alpha,\beta\in\mathcal P_{k-1}(\widehat T)$.
Also from the definition of $\Pi_k\u$ we have
\[
\int_{\widehat F_2} (\Pi_{k,1}\u -
\Pi_{k,2}\u)\,p_k = \int_{\widehat F_2}
(u_1-u_2)\,p_k\qquad \forall p_k\in\mathcal P_k(\widehat F_2),
\]
and then, if $\Pi_{k,2}\u=0$, we would obtain
\[
\int_{\widehat F_2}x_1(\alpha-x_1^k)\,p_k=0 \qquad\forall
p_k\in\mathcal P_k(\widehat F_2).
\]
But taking $p_k=\alpha(x_1,x_1,x_3)-x_1^k$ in the last equation, we
arrive at $\alpha(x_1,x_1,x_3)=x_1^k$, but this contradicts the
fact that $\alpha\in\mathcal P_{k-1}(\widehat T)$. Then we have
$\Pi_{k,2}\u\ne0$.\end{proof}

\setcounter{equation}{0}
\section{Error estimates for RT interpolation}

We end the paper giving optimal error estimates for Raviart-Thomas
interpolation of any order. These estimates are derived from the stability
results obtained in the previous sections combined with
polynomial approximation results.

Let us recall some well known properties of the averaged Taylor polynomial.
For a convex domain $D$ and any non-negative integer
$m$, given $f\in W^{p,m+1}(D)$ the averaged Taylor polynomial
is given by
$$
Q_mf(\x)=\frac1{|D|}\int_D T_mf(\y,\x)\,d\y \, ,
$$
where
$$
T_mf(\y,\x)=\sum_{|\alpha|\le m} D^\alpha
f(\y)\frac{(\x-\y)^\alpha}{\alpha!}\, .
$$
Then, there exists a constant $C$, depending only on $m$
and $D$ (see for example \cite{BS,D3}), such that
\begin{equation}
\label{ineqinterp}
\|D^\beta (f-Q_mf)\|_{\Lp(D)}
\le C \sum_{i_1+i_2+i_3=m+1}
\left\|\frac{\partial^{m+1}f}{\partial x_1^{i_1}\partial x_2^{i_2}
\partial x_3^{i_3}} \right\|_{\Lp(D)}
\end{equation}
whenever $0\le |\beta|\le m+1$.

As a consequence of these results we have the following approximation
result for elements satisfying the regular vertex property. Given a
function $f$, $D^m f$ denotes the sum of the absolute values of all the
derivatives of order $m$ of $f$.

\begin{lemma}
\label{averaged taylor}
Let $T$ be a tetrahedron satisfying \textsl{RVP}$(\bar c)$ such that
${\bf p}_0$ is the regular vertex, $\l_i, i=1,2,3$ are unitary
vectors with the directions of the edges sharing ${\bf p}_0$,
$h_i, i=1,2,3$, the lengths of these edges, and $h_T$ the diameter of $T$.
Then, given $\u\in W^{m+1,p}(T)^3$, $m\ge0$, there exists $\q\in{\mathcal P}_m(T)^3$
such that,
\begin{equation}
\label{at1}
\left\|\frac{\partial(\u - \q)}{\partial \l_1}\right\|_{\Lp(T)}
\le C \sum_{i_1+i_2+i_3=m} h_1^{i_1} h_2^{i_2} h_3^{i_3}
\left\|\frac{\partial^{m+1}\u}{\partial\l_1^{i_1+1}\partial\l_2^{i_2}
\partial\l_3^{i_3}} \right\|_{\Lp(T)}
\end{equation}
and analogously for $\frac{\partial(\u - \q)}{\partial \l_j}$ with $j=2,3$.
Also,
\begin{equation}
\label{at2}
\|\d(\u-\q)\|_{\Lp(T)}\le C h_T^m \|D^m\d\u\|_{\Lp(T)}
\end{equation}
where the constant $C$ depends only on $m$ and $\bar c$.
\end{lemma}

\begin{proof} To simplify notation we assume again that ${\bf p}_0={\bf 0}$.
From Theorem \ref{sobrervp} we know that there exists a matrix $M$
such that its associated linear transformation maps $\tT$ onto
$T$, where $\tT$ is the element with vertices at  ${\bf 0}$,
$h_1{\bf e}_1$, $h_2{\bf e}_2$ and $h_3{\bf e}_3$. Moreover, the
norms of $M$ and of its inverse matrix are bounded by a constant
which depends only on $\bar c$.

Now, as in the proof of Theorem \ref{mainrvp}, we define $\widetilde\u\in W^{m+1,p}(\tT)^3$
via the Piola transform and
$$
\btQ_m\tu=(\tQ_m \tilde u_1,\tQ_m \tilde u_2,\tQ_m \tilde
u_3)\in\mathcal P_m(\tT)^3,
$$
where $\tQ_m\tilde u_j$ denotes the averaged Taylor polynomial of
$\tilde u_j$.

Using the estimate (\ref{ineqinterp}) on the reference element $\hT$ which
has vertices at ${\bf 0}$, ${\bf e}_1$, ${\bf e}_2$ and ${\bf e}_3$, and a standard
scaling argument we obtain
$$
\left\|\frac{\partial(\tu-\btQ_m\tu)}{\partial\tilde
x_1}\right\|_{\Lp(\tT)} \le C \sum_{i_1+i_2+i_3=m} h_1^{i_1}
h_2^{i_2} h_3^{i_3} \left\|\frac{\partial^{m+1}\tu}{\partial
\tilde x_1^{i_1+1}\partial \tilde x_2^{i_2}
\partial \tilde x_3^{i_3}} \right\|_{\Lp(\tT)}.
$$
Then, defining $\q\in\mathcal P_m(T)^3$ via the Piola transform,
that is,
$$
\q(\x)=\frac1{\det M}M\btQ_m\tu(\widetilde\x),
\qquad\x=M\widetilde\x\, ,
$$
(\ref{at1}) follows by changing variables as in the proof
of Theorem \ref{mainrvp}.

On the other hand, since
$$
\widetilde\d\btQ_m\tu= \tQ_{m-1}\widetilde\d\tu ,
$$
using again (\ref{ineqinterp}) in $\hT$ and a scaling argument we obtain,
$$
\|\widetilde\d(\tu-\btQ_m\tu)\|_{\Lp(\tT)}
\le C h_{\tT}^m \|\widetilde D^m\widetilde\d\u\|_{\Lp(\tT)}
$$
and therefore, (\ref{at2}) follows by using the properties
of the Piola transform stated in (\ref{propiedadpiola}).
\end{proof}

We can now state and prove optimal error estimates for
elements satisfying the regular vertex property. Our theorem
generalizes the results proved in \cite{AD1}, where the same error estimate
was proved in the case $k=0$, as well as those proved in \cite{DL3},
where the estimate was proved for any $k\ge 0$ but only in the case $m=k$.

\begin{theorem}
Let $k\ge0$ and $T$ be a tetrahedron satisfying \textsl{RVP}$(\bar c)$. If
${\bf p}_0$ is the regular vertex, $\l_i, i=1,2,3$ are unitary
vectors with the directions of the edges sharing ${\bf p}_0$,
$h_i, i=1,2,3$, the lengths of these edges, and $h_T$ the diameter of $T$ then,
there exists a constant $C$ depending only on $k$ and $\bar c$ such that,
for $0\le m\le k$, $1\le p\le \infty$, and $\u\in W^{m+1,p}(T)^3$,
\begin{eqnarray*}
\lefteqn{\|\u - \Pi_k\u\|_{\Lp(T)}}\\&\le&
C\left\{\sum_{i_1+i_2+i_3=m+1} h_1^{i_1} h_2^{i_2} h_3^{i_3}
\left\|\frac{\partial^{m+1}\u}{\partial\l_1^{i_1}\partial\l_2^{i_2}
\partial\l_3^{i_3}} \right\|_{\Lp(T)}
 + h_T^{m+1} \|D^m\d\u\|_{\Lp(T)}\right\}
\end{eqnarray*}
\end{theorem}

\begin{proof} Since $m\le k$, for any $\q\in {\mathcal P}_m(T)^3$ we have
$$
\u - \Pi_k\u=\u-\q - \Pi_k(\u-\q)
$$
and therefore, applying Theorem \ref{mainrvp}, we obtain
\begin{eqnarray*}
\lefteqn{\left\|\u-\Pi_k\u\right\|_{\Lp(T)}} \\
&\le& C\left\{ \|\u-\q\|_{\Lp(T)} +
\sum_{i,j} h_j\left\|\frac{\partial (u_i-q_i)}{\partial
\ell_j}\right\|_{\Lp(T)} + h_T\|\d(\u-\q)\|_{\Lp(T)}\right\}.
\end{eqnarray*}
Then, taking $\q\in {\mathcal P}_m(T)^3$ satisfying (\ref{at1}) and (\ref{at2})
we conclude the proof. \end{proof}

Also optimal error estimates under the maximum angle condition
can be proved. We state the results in the following theorem.

\begin{theorem}
Let $k\ge0$ and $T$ be a tetrahedron with diameter $h_T$
satisfying \textsl{MAC}$(\bar\psi)$. There exists a constant $C$
depending only on $k$ and $\bar\psi$ such that,
for $0\le m\le k$, $1\le p\le \infty$, and $\u\in W^{m+1,p}(T)^3$,
$$
\|\u-\Pi_k\u\|_{\Lp(T)}\le C h_T^{m+1}\|D^{m+1}\u\|_{\Lp(T)}.
$$
\end{theorem}

The proof is analogous to that of the previous theorem, using now
the stability estimates obtained in Theorem \ref{mainmac}, and so
we omit the details.

\end{document}